\documentclass[a4paper]{amsart}

\newcommand{\ys}{y^{\sigma}}
\newcommand{\zs}{z^{\sigma}}
\newcommand{\abs}[1]{\left\vert#1\right\vert}

\newcommand{\N}{\mathbb{N}}

\newcommand{\set}[1]{\left\{#1\right\}}
\newcommand{\norm}[2]{\| #1 \|_{#2}}
\newcommand{\scalar}[2]{\langle{ #1},{#2} \rangle}
\newcommand{\lr}[1]{\left( #1\right)}
\newcommand{\lb}[1]{\left[ #1\right]}
\newcommand{\sset}{A_\varphi}
\newcommand{\abold}{\mathbf{a}}
\newcommand{\E}[1]{\mathbb{E}[#1]}

\theoremstyle{theorem}
\newtheorem*{theorem}{Theorem}
\newtheorem*{corollary}{Corollary}
\newtheorem{lemma}{Lemma}
\theoremstyle{definition}
\newtheorem{assumption}{Definition}
\newtheorem{remark}{Remark}
\title[statistical inverse problems under general
  source conditions]{Minimax rates for statistical inverse problems under general
  source conditions} 

\author{Litao Ding}
\email{13110180018@fudan.edu.cn (LiTao Ding)}
\address{School of Mathematical Sciences, Fudan University, Shanghai,
  China 200433}
\author{Peter Math{\'e}}
\email{peter.mathe@wias-berlin.de (Peter Math\'e)}
\address{Weierstrass Institute, Mohrenstrasse 39, D-10117 Berlin, Germany}
\begin{document}



\begin{abstract}
 We describe the minimax reconstruction rates in linear ill-posed
 equations in Hilbert space when smoothness is given in terms of
 general source sets. The underlying fundamental result, the minimax
 rate on ellipsoids, is proved similarly to the seminal study by
 D. L. Donoho, R.~C. Liu, and B. MacGibbon, {\it Minimax risk over
   hyperrectangles, and implications}, Ann.~ Statist. 18, 1990. These
 authors highlighted the special role of the truncated series
 estimator, and for such estimators the risk can explicitly be given.
We provide several examples, indicating results for
statistical estimation in ill-posed problems in Hilbert space.
\end{abstract}

\maketitle
\section{Introduction}
\label{sec:intro}

We consider linear operator equations of the form
\begin{equation}
  \label{eq:base}
\ys=Tx+\sigma\xi,
\end{equation}
where $T$ is a compact linear operator between Hilbert spaces $X$ and
$Y$ under Gaussian white noise~$\xi$ and with noise level~$\sigma>0$.
The compact operator T has a singular value decomposition, where
$\{s_n^2\}$ denotes the sequence of eigenvalues of $T^*T$, arranged in
decreasing order, and $\{v_n\}$ in $X$,
and $\{u_n\}$  in $Y$ are orthonormal systems. In particular, we have
\begin{equation}
  \label{eq:svd}
T x=\Sigma_{j=1}^{\infty}s_j\scalar{x}{v_j}u_{j},\quad \quad  x\in X.
\end{equation}
By using the singular value decomposition, and letting~$\sigma_{k}:=
s^{-1}_{k},\ k=1,2,\dots$, as well as the
coefficients~$\theta_k=(x,v_k),\ k=1,2,\dots$ we find that the
model~(\ref{eq:base}) is equivalent to the sequence space model
\begin{equation}
  \label{eq:seq-model}
 \zs_{k} =\theta_k+\sigma\sigma_k\xi_k ,\quad\quad \xi_k\sim N(0,1),
\end{equation}
provided that the solution element~$x$ is in the orthogonal complement
of the kernel of the operator~$T$.
Notice that as a consequence of the compactness of the operator~$T$ we
have that~$s_{k}\to 0$ as $k\to \infty$, and hence~$\sigma_{k}\to
\infty$, such that higher number coefficients are blurred by higher
noise level. This is a typical feature of inverse problems, and it
thus requires to regularize the observations~$\zs_{k},\ k=1,2,\dots$

In order to apply the minimax paradigm for the analysis of statistical
inverse problems we introduce classes of
elements~$\theta = \lr{\theta_{k}}_{k=1}^{\infty}$.
Prototypical classes are given through \emph{Sobolev-type} ellipsoids.
\begin{assumption}
  [Sobolev-type ellipsoid]
For a given increasing sequence~$\abold = \lr{a_{j}}_{j=1}^{\infty},\ a_{j} > 0$, and a constant~$Q< \infty$
we let
\begin{equation}
  \label{eq:sobolev}
  \Theta_{\abold}(Q) = \set{\theta,\quad  \sum_{j=1}^{\infty} a_{j}^{2}
    \theta_{j}^{2}\leq Q^{2}}.
\end{equation}
\end{assumption}
The increasing sequence~$a_{j}>0$ controls the decay of the
coefficients of~$\theta$. Within the theory of inverse problems it is
common to define the ellipsoid \emph{relative to the operator~$T$} in
the following form.
\begin{assumption}[General source set]
 {
 For a continuous non-decreasing function $\varphi$ with $\varphi(0)=0$ we let
 \begin{equation}
   \label{eq:source-set}
\sset=\{x\in X,\quad  x=\varphi(T^*T)v,||v||_2\le 1\}.
 \end{equation}
}
\end{assumption}
These ellipsoids are related as follows. Since for~$x\in\sset$ we have
that~$\theta_{j}= \scalar{x}{v_{j}}=
\varphi(s_{j}^{2})\scalar{v}{v_{j}},\ j\in\N$, we find that
$$
\norm{v}{2}^{2}:= \sum_{j=1}^{\infty}\frac{\theta_{j}^{2}}{\varphi(s_{j}^{2})},
$$
such that~$x\in\sset$ implies for~$a_{j}:=
\frac{Q}{\varphi(s_{j}^{2})},\ j\in\N$ that~$\theta\in
\Theta_{\abold}(Q)$. More specific examples will be given in
Section~\ref{sec:examples}.

The goal of the present study is to establish minimax rates for
statistical estimation over such sets of elements. Let~$\hat \theta: =
\hat\theta(\zs)$ be any estimator based on the data~$\zs$. Its RMS-error
on the class~$\Theta_{\abold}$ is then given as
\begin{equation}
  \label{eq:stat-error}
  e(\hat\theta, \Theta_{\abold},\sigma) :=\sup_{\theta\in \Theta_{\abold}} \lr{\E{\norm{\theta - \hat \theta}{X}^{2}}}^{1/2}.
\end{equation}
Above, we highlight that the error depends on the underlying noise level~$\sigma$.
The minimax error on a class~$\Theta_{\abold}$ is consequently given
as
\begin{equation}
  \label{eq:minimax-error}
  e(\Theta_{\abold},\sigma):= \inf_{\hat{\theta}}e(\hat\theta, \Theta_{\abold},\sigma).
\end{equation}
We will denote by
~$e_{T}(\Theta_{\abold},\sigma)$ the minimax
errors when restricted to 
truncation estimators, only. 
It will be handy to introduce the following quantity
\begin{equation}
  \label{eq:rhoD}
  \rho_{n}^{2}:= \sum_{j=1}^{n} \frac{1}{s_{j}^{2}},\quad n\in\N.
\end{equation}

The main result presented in this note is the following
\begin{theorem}
  Let~$\Theta_{\abold}$ be as in~(\ref{eq:sobolev}), and let~$s_{j},\
  j=1,2,\dots$ denote the singular numbers of the operator~$T$. Then
  for~$\sigma>0$ we have that
  \begin{equation}
    \label{eq:minimax-abold}
    e(\Theta_{\abold},\sigma) \leq \inf_{D}\lr{
      \frac{Q^{2}}{a_{D+1}^{2}} + \sigma^{2} \rho_{D}^{2}}^{1/2} \leq 2.2  e(\Theta_{\abold},\sigma)
  \end{equation}
\end{theorem}{}

We formulate the following consequence for source sets.
\begin{corollary}
  Consider the source set~(\ref{eq:source-set}) from above. Then
  \begin{equation}
    \label{eq:minimax-sset}
e^{2}(\sset,\sigma) \leq \inf_{D} \set{\varphi^{2}(s_{D+1}^{2}) + \sigma^{2}
  \rho_{D}^{2}}\leq 4.84  e^{2}(\sset,\sigma).
  \end{equation}
\end{corollary}
This consequence allows to explicitly compute the minimax rates for
statistical inverse problems in Hilbert space for a variety of index
functions~$\varphi$ and decay rates of singular numbers~$s_{j},\
j\in\N$ in subsequent examples. 

\section{Proof of the theorem}
\label{sec:proof}

The main portion in this note is to provide the details for proving
the main result. The fundamental underlying principle is comprised in
three observations.
\begin{enumerate}
\item The error   of the truncated (at $D$-th summand) series estimator upper bounds the
  minimax error.
\item The intermediate $\inf$ in~(\ref{eq:minimax-abold}) is the error
  of the (best) truncated series estimator.
\item Up to a factor $2.22$ the (squared)  error   of the truncated series
  estimator is best possible.
\end{enumerate}
The first assertion is trivial.
For the second assertion let us introduce the truncated series
estimator, given observations~$\zs_{j},\ j\in\N$, as
\begin{equation}
  \label{eq:ts-estimator}
  \hat \theta_{n}(\zs) := \sum_{j=1}^{n} \zs_{j} v_{j}.
\end{equation}
\begin{remark}
  The above estimator corresponds to the spectral cut-off estimator
  for the original data~$\ys$. In these terms  we find that
$$
 \hat \theta_{n}(\zs) =  \sum_{j=1}^{n}\frac 1 {s_{j}} \ys_{j} v_{j}.
$$
\end{remark}
\begin{lemma}\label{lem:qt-estimator}
  We have that
$$
e^{2}(\hat \theta_{n},\Theta_{\abold},\sigma) =
\frac{Q^{2}}{a_{n+1}^{2}} + \sigma^{2} \rho_{n}^{2}.
$$
\end{lemma}
\begin{proof}
  We derive the bias-variance decomposition for the estimator as
  \begin{align*}
  \E{\norm{\theta - \hat \theta_{n}(\zs)}{X}^{2}} & =
                                                    \sum_{j=1}^{n}\E{\lr{\theta_{j}
                                                    - \zs_{j}}^{2}} +
                                                    \sum_{j=n+1}^{\infty}
                                                    \theta_{j}^{2}\\
&=  \sigma^{2}\sum_{j=1}^{n}\frac{1}{s_{j}^{2}} + \sum_{j=n+1}^{\infty}
                                                    \theta_{j}^{2}.
  \end{align*}
Thus, uniformly over the class~$\Theta_{\abold}$ we find that
$$
\sup_{\theta\in\Theta_{\abold}} \E{\norm{\theta - \hat
    \theta_{n}(\zs)}{X}^{2}}
=  \sigma^{2}\sum_{j=1}^{n}\frac{1}{s_{j}^{2}} +
\sup_{\theta\in\Theta_{\abold}}\sum_{j=n+1}^{\infty} \theta_{j}^{2}
$$
Now we observe that
$$
\sup_{\theta\in\Theta_{\abold}}\sum_{j=n+1}^{\infty} \theta_{j}^{2}
=
\sup_{\theta\in\Theta_{\abold}}\sum_{j=n+1}^{\infty}a_{j}^{2}\theta_{j}^{2}
a_{j}^{-2}\leq \frac{Q^{2}}{a_{n+1}^{2}},
$$
from the monotonicity assumptions for~$\abold$. But the upper bound is
attained for the element
$$
\theta^{0}_{j} :=
\begin{cases}
  \frac{Q}{a_{n+1}}&, j=n+1\\
0 &, \text{else.}
\end{cases}
$$
The proof of the Lemma is complete.
\end{proof}
\begin{remark}
  The estimator~$\hat\theta_{n}$, which selects the first $n$
  coordinates,  is the best among all truncation
  estimators which use $n$ coordinates. Suppose that a truncation
  estimator, say $\hat \theta_{P}$ uses a
  set~$P\subset\set{1,2,\dots},\ |P|=n$. Then its squared risk
  (uniform on~$\Theta_{\abold}$) is given as~$\sup_{\theta\in
    \Theta_{\abold}}\sum_{j\not \in P}\theta_{j}^{2} + \sigma^{2}\sum_{j\in P}\frac
  1 {s_{j}^{2}}$. Obviously, the variance term is mimimal for the
  initial segment. If~$k := \min P^{c} \leq n$, then similar to above,
  we let~$\theta^{0}$ have only the~$k$th component different from zero
  with~$\theta_{k}^{0}:= Q/a_{k}$. Then the uniform squared bias is upper bounded
  by its value at~$\theta^{0}$, and this gives~$Q^{2}/a_{k}^{2} \geq
  Q^{2}/a_{n+1}^{2}$, from monotonicity. Therefore, consideration may
  be restricted to the initial segment, only.
\end{remark}
We turn to proving the final assertion from above, and this is the
crucial part for the overall proof. To this end we use arguments
from the seminal study~\cite{MR1062717}. In Corollary (to Theorem~10)
ibid. these authors establish that for \emph{orthosymmetric, compact,
  convex and quadratically convex} sets~$\Theta$ the squared risk of the best
truncated series estimator is less than 4.44 times the squared minimax risk
over~$\Theta$. Ellipsoids as the set~$\Theta_{\abold}$ are
prototypical examples of such sets.

However, in that study the model was the sequence space model, similar
to~(\ref{eq:seq-model}),  with i.i.d
noise~$\xi_{j}$, but  with~$\sigma_{j}\equiv 1$. Thus their argument
does not cover statistical inverse problems (with compact operator~$T$
as in~(\ref{eq:base})).
Therefore, we (briefly) recall the arguments used in the study~\cite{MR1062717}.

The proof can be divided as follows.
Let~$\Theta(r) := \set{\theta,\ \abs{\theta_{i}}^{2}\leq
  r_{i}}\subset\Theta_{\abold}$ ($r_{i}\to 0$) be any hyperrectangle. 
\begin{enumerate}
\item Plainly we have that
  \begin{equation}
    \label{eq:plainly}
e(\Theta(r),\sigma) \leq e(\Theta_{\abold},\sigma) \leq e_{T}(\Theta_{\abold},\sigma).
  \end{equation}
\item The reasoning in~\cite{MR1062717} starts with considering
  1d-subproblems, and there we cannot distinguish between regression
  or inverse estimation problems. It is shown
  in~\cite[Thm.~2]{MR1062717} that for 1d-problems best nonlinear
  estimators cannot perform better that $\sqrt{2.22}$ times best truncation
  estimators, or the best linear estimator.
\item This extends to hyperrectangles, as shown in ~\S~3 ibid, because
  the minimax risk is a Bayes  risk, and the worst Bayes prior is of
  product type. Specifically, as shown in~\cite[Prop.~8]{MR1062717} we find that~$e_{T}(\Theta(r),\sigma) \leq \sqrt{2.22}
  e(\Theta(r),\sigma)$, thus, together with~(\ref{eq:plainly}) we already find
  \begin{equation}
    \label{eq:hyper-ineq}
  e_{T}(\Theta(r),\sigma) \leq \sqrt{2.22}   e_{T}(\Theta_{\abold},\sigma).
  \end{equation}
\item In  view of Lemma~\ref{lem:qt-estimator} the assertion of the
  theorem follows once we can prove the next result, similar to~\cite[Thm.~10]{D/L/M} (Notice
  that~$\sqrt{2.22} * \sqrt 2 \leq 2.2$).
\end{enumerate}

  \begin{lemma} For ellipsoids we have that
$$
e_{T}(\Theta_{\abold},\sigma) \leq \sqrt 2 \sup\set{e_{T}(\Theta(r),\sigma),\quad
  \Theta(r) \subset \Theta_{\abold})}.
$$
  \end{lemma}
  \begin{proof}
Arguing as in~\cite{MR1062717} for ellipsoids the error is controlled
on the positive orthant~$\Theta_{+}$.
    We consider the following functional, which expresses the
      minimax risk of a truncated series estimator on the
      hyperrectangle~$\Theta(r)$.
    \begin{equation}
      \label{eq:J}
      J(r) := \sum_{i=1}^{\infty}
      \min\set{r_{i},\sigma^{2}\sigma_{i}^{2}},\quad r\in \Theta_{+}^{2}.
    \end{equation}

This functional is concave (as a sum of concave functions), and it is
continuous, such that on the compact set~$\Theta_{+}^{2}$ it attains it's
maximum value, say at~$r^{\ast}$. This maximality property yields
that~$J(r) \leq J(r^{\ast}),\ r\in \Theta_{+}^{2}$. We introduce the
sets~$Q:= \set{i, r^{\ast}_{i} =
  \sigma^{2}\sigma_{i}^{2}}$, and the~$P:= \set{i, r^{\ast}_{i}\geq
  \sigma^{2}\sigma_{i}^{2}}$. 
The latter will be finite,
since~$r^{\ast}_{i}\to 0$ and~$\sigma^{2}\sigma_{i}^{2}\to \infty$.
With this notation we can derive, similarly as in the original
study~\cite{D/L/M},  the explicit form of the Gateaux
derivative of the functional~$J$ at~$r^{\ast}$ as
\begin{equation}
  \label{eq:gateaux}
  D_{r^{\ast}}J(h) = \sum_{i\not\in P} h_{i} - \sum_{i\in Q}
  (h_{i})_{-},\quad  h := r - r^{\ast}\in \Theta_{+}^{2}.
\end{equation}
Observing that for~$i\in Q$ and~$h_{i}<0$ we have
\begin{equation}
  \label{eq:h-}
  (h_{i})_{-} = - h_{i} = r^{\ast}_{i} - r_{i} \leq \sigma^{2}\sigma_{i}^{2},
\end{equation}
and using that~$D_{r^{\ast}}J(h) \leq 0$ (maximality property) we
arrive at
\begin{equation}
  \label{eq:ri-ineq}
\sum_{i\not \in P}r_i<\sum_{i\not\in P}r^{\ast}_{i} + \sigma^2\sum_{i\in
  Q}\sigma_i^2,\quad r\in \Theta_{+}^{2}.
\end{equation}
Consider the truncation estimator $\hat{\theta}^{*}$ for
$\Theta_{\abold}$ given by $\hat{\theta}^{*} =\zs_i\chi_{(i\in
  P)}$. It has the following (squared) risk
\begin{align*}
\E{\norm{\hat{\theta}^{*} - \theta}{}^{2}} &=\sum_{i\not\in
  P}\theta_i^2+ \sigma^2\sum_{i\in P}\sigma_i^2.
\end{align*}
Since by definition~$Q \subset P$, and with~$r_{i}:= \theta_{i}^{2}$
this can be upper bounded by
$$
\E{\norm{\hat{\theta}^{*} - \theta}{}^{2}}
\le\sum_{i\not \in
  P}r^{\ast}_i +2\sigma^2\sum_{i\in P}\sigma_i^2\le2 J(r^{\ast}).
$$
Therefore,~$e_{T}^{2}(\Theta_{\abold},\sigma) \leq \E{\norm{\hat{\theta}^{*} -
    \theta}{}^{2}} \leq 2 J(r^{\ast})$.
Since, by construction of~$r^{\ast}$ we have that~$J(r^{\ast}) =  \sup\set{e_{T}(\Theta(r),\sigma),\quad
  \Theta(r) \subset \Theta_{\abold})}$ the proof of the lemma is complete.
  \end{proof}

\section{Discussion and examples}
\label{sec:examples}

Minimax rates for statistical inverse problems were established in
many studies. The seminal study is~\cite{MR1409127}. An update of the
subsequent studies was given in the survey~\cite{MR2421941}, which will be used
for comparison below. Two studies are related to questions in the
present study.

The study~\cite{MR2361904} provides a detailed analysis of statistical
inverse problems (in a slightly more general framework). Theorem~8
ibid. actually asks for the relation between the minimax error and best
error bounds for spectral cut-off. Assumptions are given where the
error of a certain specific regularized estimator is not worse than the error
obtained by spectral cut-off. In our approach we do not rely on any
specific way of obtaining estimators of the unknown element. 

The authors in~\cite{loubes2009} discuss, among
many other things, the relation between ellipsoids and source
sets as means of regularity conditions. Theorem~4.1 ibid. asserts
the minimax rate for power type decay of singular numbers of the operator~$T$ and power
type increase of the sequence~$\abold$ in the
ellipsoid~(\ref{eq:sobolev}). One specific focus is on the concept of
maxisets, which is an interesting approach, but which is beyond the
focus of the present study.

Previous minimax rates on general
source sets were given in~\cite{MR2240642}, but restricted to
\emph{operator monotone} index functions, only. These index functions
are limited to low smoothness, and the result was obtained by an
application of Pinsker's study~\cite{MR82j:93048}.

Next, we relate the results to similar results as known in 'classical'
inverse problems and to results from non-parametric statistical
testing. 
 For classical inverse problems, when the noise
 obeys~$\norm{\xi}{Y}\leq 1$, then the minimax rate of recovery is 
 related to the \emph{modulus of continuity}. We skip details and
 refer to  the study~\cite{MR1984890}. As can be seen from those
 results, the corresponding minimax rate is (up to the constant~$\sqrt
 2$) given by $\inf_{D}\set{\frac{Q^{2}}{a_{D+1}^{2}} + \sigma^{2}
  \frac{1}{s_{D}^{2}}}^{1/2}$, which, again is attained by a
truncated series estimator. Looking the the second summand above we
immediately conclude that for exponentially decaying singular numbers
the rates for 'classical' and statistical inverse problems are the same.

 In non-parametric statistical testing such
  estimators also play a similar role, and (the square of) the  \emph{minimum
    separation radius} was given in~\cite[Prop.~3]{MR2879673} as
$\inf_{D}\max\set{\frac{Q^{2}}{a_{D+1}^{2}}, \sigma^{2}
  \lr{\sum_{j=1}^{D} \frac{1}{s_{j}^{4}}}^{1/2}}$. Additional
constants, determined by the prescribed errors of the first and second
kind appear. Still, the truncated series estimator plays a prominent
role. The major differences can be seen from the significance of the
fourth powers instead of the second ones, such that 'estimation is
harder than testing'. Again, for exponentially decaying singular
numbers the rates for testing and estimation coincide.

\medskip

The following examples highlight the minimax rates in statistical
inverse problems in prototypical situations. The corresponding rates are
also seen in~\cite[Tbl.~1]{MR2421941}. 
\begin{assumption}[ill-posedness of the operator]
We call the operator~$T$ \emph{modeately ill-posed} if its singular
numbers decay at a power type rate~$s_{j}\asymp  j^{-p},\
j=1,2,\dots$ for some $p>0$. It is called 
\emph{severely ill-posed} if~$s_{j}\asymp e^{-pj},\ j=1,2,\dots$  
\end{assumption}
\begin{assumption}
  [solution smoothness]
The solution smoothness, expressed in terms of~$\Theta_{\abold}$ is
said to be~\emph{moderate} if $a_{j}\asymp j^{\kappa}$. It is called
\emph{analytic} in case~$a_{j}\asymp e^{\kappa j}$.
\end{assumption}
Finally, we introduce the corresponding notion for the statistical
problem at hand.
\begin{assumption}[ill-posedness of the statistical problem]
  The statistical problem~(\ref{eq:base}) is called \emph{moderately
    ill-posed} if the minimax rate of reconstruction ( in terms of the
  noise level~$\sigma$) is of power
  type. It is called \emph{severely ill-posed} if the rate is
  \emph{logarithmic}, and it is called \emph{mildly ill-posed} if (up
  to a logarithmic factor) the minimax rate is linear in the noise level~$\sigma$.
\end{assumption}


The results as obtained from the application of the theorem are shown
in  Table~\ref{tab:rates-cavalier}.  
The number~$D_{\ast}$ denotes the optimal
  truncation level, the number which balances both terms in the middle
  sum in~(\ref{eq:minimax-abold}). By 'rate' we denote the
  corresponding minimax rate as~$\sigma \to 0$.

\begin{table}[ht]
  \centering
  \begin{tabular}{l| rl | rl }
\hline
 & $s_{j} \asymp j^{-p}$ 
& $\rho_{D}^{2} \asymp D^{2p+1}$ 
&  $s_{j}\asymp e^{-pj}$ 
&  $\rho_{D}^{2} \asymp e^{2p D}$\\\hline
 & $\varphi(t)=$ & $t^{\kappa/(2p)}$  
& $\varphi(t)=$  
& $ \log^{-\kappa}(1/t)$\\
$a_{j}\asymp j^{\kappa}$    
& $D_{\ast} = $ 
&  $\lr{\frac{1}{\sigma}}^{1/(\kappa + p +1/2)}$ 
& $D_{\ast}= $ 
&$\log(1/\sigma)$\\
& rate:
& $\sigma^{\kappa/(\kappa + p + 1/2)}$
& rate:
& $\log^{-\kappa}(1/\sigma)$ \\\hline
&$\varphi(t)=$
&  $e^{- \kappa t^{-1/(2p)}}$
& $\varphi(t)=$
& $t^{\kappa/(2p)}$\\
$a_{j}\asymp e^{\kappa j}$
& $D_{\ast} =$
&$\frac{1}{2\kappa} \log(1/\sigma)$
& $D_{\ast} = $
& $ \frac 1 {p + \kappa} \log(1/\sigma) $ \\
& rate:
&  $ \sigma \lb{\log(1/\sigma)}^{{p+1/2}}$ 
& rate:
& $\sigma^{\kappa/(p + \kappa)}$\\\hline
  \end{tabular}
  \vspace*{.5\baselineskip}
  \caption{\small Outline of reconstruction rates for Sobolev-type smoothness of
    the truth~$\theta\in\Theta_{\abold}$ and power/exponential type decay of the
    singular numbers.}
  \label{tab:rates-cavalier}
\end{table}
As we see, we find mild ill-posedness of the problem for moderately
ill-posed operator and analytic smoothness. The problem is severely
ill-posed for severely ill-posed operator and moderate smoothness. We
emphasis that the problem is moderately ill-posed, both for moderate
ill-posedness of the operator and power type smoothness, but also for
analytic smoothness and severely ill-posed operator. That is why the
distinction between ill-posedness of the operator and ill-posedness of
the statistical problem is recommended.

\bibliographystyle{plain}
\bibliography{minimax}




\end{document}